\documentclass{amsart}
\usepackage{amssymb,latexsym}
\usepackage{amsmath}
\usepackage{amscd}
\usepackage{graphicx}
 \usepackage{color}
\usepackage{enumerate}
\numberwithin{equation}{section}
\theoremstyle{plain}
 \newtheorem{theorem}{Theorem}[section]
 \newtheorem{lemma}[theorem]{Lemma}
 
 \newtheorem{observation}[theorem]{Observation}
 
\theoremstyle{definition}
 \newtheorem{definition}[theorem]{Definition}

\newenvironment{enumeratei}{\begin{enumerate}[\quad\upshape (i)]} {\end{enumerate}}

\newcommand \datum {February 18, 2018} 
\newcommand \ncross {\eqref{eqpbxcspLndf}-crosses}

\newcommand \notncross {does not \eqref{eqpbxcspLndf}-cross}

\newcommand \sncross {strongly \eqref{eqpbxcspLndf}-cross each other}
\newcommand \wncross {weakly \eqref{eqpbxcspLndf}-cross each other}
\newcommand \notwncross {do not \eqref{eqpbxcspLndf}-cross each other weakly}
\newcommand \ocrossing {Fejes-T\'oth-crossing}
\DeclareMathOperator\Con{Con}
\newcommand \CrCo {\textup{HCC}}
\newcommand \rotleq {\leq_{\textup{rot}}}
\newcommand \rotnleq {\nleq_{\textup{rot}}}

\renewcommand \phi {\varphi}
\renewcommand \epsilon {{\boldsymbol\varepsilon}}
\newcommand \ucirc {C_{\kern-1pt\textup{unit}}}
\newcommand \bnd [1] {\partial #1} 
\newcommand \lne {\ell}

\newcommand \slcr [1] {\textup{Sli}(#1)}

\newcommand \cylin {\textup{Cyl}}
\newcommand \dir [1] {\textup{dir}(#1)}
\newcommand \deletethis [1] {}  
\newcommand \phullo {\textup{Conv}} 
\newcommand\set [1]{\{#1\}}
\newcommand \tuple [1] {\langle #1 \rangle}
\newcommand \pair [2] {\tuple{#1,#2}}
\newcommand \real {\mathbb R}

\newcommand \preal {\real^{2}}
\newcommand \red [1] {{\color{red}#1\color{black}}}
\newcommand \tbf[1] {\textbf{#1}}

\begin{document}
\title
{Circles and crossing planar compact convex sets}

\author[G.\ Cz\'edli]{G\'abor Cz\'edli}
\email{czedli@math.u-szeged.hu}
\urladdr{http://www.math.u-szeged.hu/\textasciitilde{}czedli/}
\address{University of Szeged\\ Bolyai Institute\\Szeged,
Aradi v\'ertan\'uk tere 1\\ Hungary 6720}

\thanks{This research was supported by
the Hungarian Research Grant KH 126581}

\begin{abstract} 
Let $K_0$ be a compact convex subset of the plane $\preal$, and assume that whenever $K_1\subseteq \preal$ is congruent to $K_0$, then $K_0$ and $K_1$ are not  crossing in a natural sense due to  L.\ Fejes-T\'oth. A theorem of L.\ Fejes-T\'oth from 1967 states that the assumption above  holds for  $K_0$ if and only if $K_0$ is a disk. In a paper appeared in 2017, the present author introduced a new concept of crossing, and proved that L.\ Fejes-T\'oth's theorem remains true if the old concept is replaced by the new one. Our purpose is to describe the hierarchy among several variants of the new concepts and the old concept of crossing. In particular, we prove that each variant of the  new concept of crossing is more restrictive then the old one. Therefore, L.\ Fejes-T\'oth's theorem from 1967 becomes an immediate consequence of the 2017 characterization of circles   but not conversely. 
Finally, a mini-survey shows that this purely geometric paper has 
 precursor in combinatorics and, mainly, in lattice theory.
\end{abstract}



\date{\red{\datum}}

\maketitle

\section{Aim and introduction}
Denoting the (usual real) Euclidean plane  by $\preal$, let $X$ and $Y$ be subsets of $\preal$. We say that $X$ and $Y$ are \emph{congruent} (also called \emph{isometric}) if
there exists a distance-preserving bijection $\phi\colon \preal \to \preal$ such that $\phi(X)=Y$. The \emph{convex hull} $\phullo(X)$ of $X\subseteq \preal$ is the smallest convex subset of $\preal$ that contains $X$. \emph{Disks} and \emph{circles} are subsets of $\preal$ of the form
$\set{\pair x y : x^2+y^2 \leq r^2}$ and $\set{\pair x y : x^2+y^2 = r^2}$ where $r\in \real$, respectively; they are necessarily nonempty sets.

There is an everyday but not precise meaning of the clause that ``two congruent convex subsets $X$ and $Y$ of $\preal$ are \emph{crossing}''. For example, the ``plus'' symbol $+$ is the union of two congruent (in fact, rotated) crossing copies of the ``minus'' symbol $-$.  
Similarly, if $X$ is a convex hull of an ellipse that is not a circle and $Y$ is obtained from $X$ by rotating it around its center point by 90 degrees, then $X$ and $Y$ are crossing. In order to make a distinction from new concepts to be discussed later, we name the first
precisely defined concept of crossing after its inventor, see Fejes-T\'oth~\cite{fejestoth};  see also the review MR0226479 (37 \#2068) on  \cite{erdosstraus} in MathSciNet.

\begin{definition} Let $X$ and $Y$ be convex subsets of the Euclidean plane. We say that $X$ and $Y$ are \emph{\ocrossing{}} if none of the sets $X\setminus  Y$ and $Y\setminus X$ are  connected.
\end{definition} 

A subset $X$ of $\preal$ is \emph{connected} (in other words, \emph{path-connected}) if for any two points $A,B\in X$ there is a continuous curve $g\subseteq X$ from $A$ to $B$. In particular, the empty set is connected; so if $X$ and $Y$ are \ocrossing{}, then  $X\setminus  Y$ and $Y\setminus X$ are nonempty.
Let us recall the following theorem.

\begin{theorem}[Fejes-T\'oth~\cite{fejestoth}] \label{ftthm}
For every nonempty compact subset $X$ of the Euclidean plane $\preal$, the following two conditions are equivalent.
\begin{enumerate}[\upshape (a)]
\item\label{ftthma} There exists no $Y\subseteq \preal$ such that $Y$
is congruent to $X$ and $X$ and $Y$ are \ocrossing{}.
\item\label{ftthmd} $X$ is a disk. 
\end{enumerate}
\end{theorem}

Hence, condition \eqref{ftthma} above characterizes disks among compact subsets of $\preal$. Since circles are exactly the boundaries of disks and disks are the convex hulls of circles, 
Theorem~\ref{ftthm} gives the following \emph{characterization of circles} immediately: 
\begin{equation}
\parbox{9.9cm}{a subset $X\subseteq \preal$ is a circle if and only if $X$ is the boundary of $\phullo(X)$ and $\phullo(X)$ satisfies condition \eqref{ftthma} of Theorem~\ref{ftthm}.}
\label{eqpbxcRcRcrDs}
\end{equation}
Since \eqref{ftthmd} trivially implies \eqref{ftthma}, the essence of   Theorem~\ref{ftthm} is that  \eqref{ftthma} implies \eqref{ftthmd}.    
Note that some stronger statements are also known. It is implicit in Fejes-T\'oth~\cite{fejestoth} that 
if we replace ``is congruent to" in \eqref{ftthma} by ``is obtained by a rotation from'', then  \eqref{ftthma} becomes weaker but it still implies  \eqref{ftthmd}; in this way,  Theorem~\ref{ftthm} turns into a stronger statement. Also, Fejes-T\'oth~\cite{fejestoth} 
extends the validity of Theorem~\ref{ftthm} 
for subsets of a sphere, while 
Erd\H os and Straus~\cite{erdosstraus} 
extends the results of \cite{fejestoth} for higher dimensions.
As a by-product of a long proof given in  Cz\'edli~\cite{czgcharcircab}, we are going to cite  a statement as Theorem~\ref{thmchgcharAB} here, which looks similar to  Theorem~\ref{ftthm}. 
A new way of crossing has naturally been introduced in the above-mentioned long proof. The main result of the present paper, Theorem~\ref{thmthemain}, describes the hierarchy for the old concept and some variants of the new concept of crossing for compact convex subsets of $\preal$. As a corollary of the main result, it will appear that  Theorem~\ref{ftthm} follows trivially from Theorem~\ref{thmchgcharAB} but not conversely; see Observation~\ref{observatimpLc}.

\subsection*{Outline and prerequisites} The rest of the paper is structured as follows. In Section~\ref{sectnewcross}, we define some new concepts of crossing and formulate our main result, Theorem~\ref{thmthemain}. Section~\ref{sectCompMn} is devoted to the proof of  Theorem~\ref{thmthemain}; up to the end of this section, 
 the paper is intended to be readable for most mathematicians. 
Finally, Section~\ref{sectionoutlook} is a historical mini-survey to point out that besides geometry, this paper has
precursors in  combinatorics and, mainly, in  lattice theory; this section can be interesting mainly for those who are a bit  familiar  with the mentioned fields.

\section{New concepts of crossing and our main result}\label{sectnewcross}
First, we recall some notations, well-known concepts, and well-known facts from Cz\'edli~\cite{czgcharcircab} and Cz\'edli and Stach\'o~\cite{czgstacho}. In order to ease our terminology, let us agree that every convex set in this paper is assumed to be nonempty, even if this is not always mentioned. By a \emph{direction} we mean a point $\alpha$ on the
\begin{equation}
\text{\emph{unit circle}\quad $\ucirc:=\set{\pair x y\in\preal: x^2+y^2=1}$.}
\label{equnitCrlcle}
\end{equation}
A direction $\pair x y\in\ucirc$ is always identified with the angle $\alpha$ for which we have that $\pair x y =\pair{\cos\alpha}{\sin\alpha}$; of course, $\alpha$ is determined only modulo $2\pi$.
This convention allows us to write, say, $\pi < \dir\lne <2\pi$ instead of saying that the direction of a line $\lne$ is strictly in the lower half-plane. If $\lne_1$ and $\lne_2$ are (directed) lines
that are equal as undirected lines but their orientations are opposite, that is, $\dir{\lne_2}=\dir{\lne_1}+\pi$, then we denote $\lne_2$ by $-\lne_1$.  As another notational convention, let us agree that for points $A$ and $B$ of a line $\lne$, we write $A<B$ or $A<_\lne B$ to denote that the direction of the vector from $A$ to $B$ is the same as that of $\lne$. For example, if $\lne$ is the $x$-axis with $\dir \lne=\pair 1 0\in\ucirc$ or, in other words, $\dir\lne =0$,  then $\pair {1}0 < \pair 2 0$.
Unless otherwise stated explicitly,  
\begin{equation}
\text{every line in this paper will be \emph{directed}};
\label{eqtxtbrdRsct}
\end{equation}
we denote the direction of a line  $\lne$ by $\dir \lne\in \ucirc$. 
In our figures, the direction of a line $\lne$ is denoted by an arrowhead, and we use a half arrowhead to indicate the left half-plane determined by $\lne$.  
Let $X\subseteq \preal$ be a compact convex set.  Its  \emph{boundary}  will be denoted by $\bnd X$. An undirected line $\lne$ is an \emph{undirected supporting line} of $X$ if $\lne\cap X\neq\emptyset$ and $X$ lies in one of the closed half-planes determined by $\lne$.
\begin{equation}
\parbox{9.0cm}{A (directed) line $\lne$ is a \emph{supporting line} of $X$ if $\lne\cap X\neq\emptyset$ and $X$ lies in the \emph{left} closed half-plane determined by $\lne$.}
\label{eqpbxDsPlNdf}
\end{equation}
The properties of supporting lines that we need here are more or less clear by geometric intuition and they are discussed in  Cz\'edli and Stach\'o~\cite{czgstacho}
at an elementary level. For a more advanced treatise, one can resort to Bonnesen and Fenchel~\cite{bonnesenfenchel}. 
Two sets are \emph{incomparable} if none of them is a subset of the other. Note that 
\begin{equation}
\parbox{9.9cm}{for each $\alpha\in\ucirc$, there is a unique supporting line $\lne$ of $X$ such that $\dir\lne=\alpha$. Furthermore, any two incomparable compact convex sets $X_1$ and $X_2$ have a common \emph{directed} supporting line that is also a supporting line of $\phullo(X_1\cup X_2)$.}
\label{eqtxtuNspalpha}
\end{equation}
After \eqref{eqtxtbrdRsct}, the adjective ``directed'' above occurs only for emphasis. Note that two disjoint compact convex subsets of $\preal$ with nonempty interiors have exactly two  common supporting lines and four non-directed common supporting lines; see the second half of Figure~\ref{figfour} for an illustration. If disjointness is not stipulated, then two incomparable compact convex sets can have much more than two common supporting lines. 
By basic properties of continuous functions and since our lines are directed, if   $X_1$ and $X_2$ are compact convex subsets of $\preal$ and   $\lne$ is a common supporting line of them, then $\lne\cap(X_1\cup X_2)$, with respect to its direction $\dir\lne$, has a unique \emph{first point} and a unique \emph{last point}.

\begin{definition}\label{defnewcross} Let $D$ and $L$ be  compact convex subsets of $\preal$. We say that $D$ \emph{\ncross{}} $L$  if $D$ and $L$  have two \emph{distinct} common supporting lines $t$ and $t'$ such that
\begin{equation}
\parbox{8.0cm}{the first point $U_D$ of $(D\cup L)\cap t$ is in $D\setminus L$,\\
the last point $U_L$ of $(D\cup L)\cap t$ is in $L\setminus D$,\\
the first point $U'_D$ of $(D\cup L)\cap t'$ is in $D\setminus L$, and\\
the last point $U'_L$ of $(D\cup L)\cap t'$ is in $L\setminus D$;
}
\label{eqpbxcspLndf}
\end{equation}
where ``first'' and ``last'' refer to the orientation of the common supporting line in question. Also, we say that $D$ and $L$ \emph{\sncross{}} if $D$ \ncross{} $L$ and $L$ \ncross{} $D$. 
Finally, we say that $D$ and $L$ \emph{\wncross{}} if $D$ \ncross{} $L$ or $L$ \ncross{} $D$.
\end{definition}

Armed with Definition~\ref{defnewcross}, we recall the following statement from Cz\'edli~\cite{czgcharcircab}.

\begin{theorem}[Lemma 3.3 in \cite{czgcharcircab}] \label{thmchgcharAB}
For every nonempty compact convex subset $X$ of the Euclidean plane $\preal$, the following two conditions are equivalent.
\begin{enumerate}[\upshape (a)]
\item\label{ftthmc} There exists no $Y\subseteq \preal$ such that $X$ and $Y$ \wncross{} and, in addition,  $Y$
is congruent to $X$.
\item\label{ftthmd} $X$ is a disk. 
\end{enumerate}
\end{theorem}

Next, we clarify the hierarchy of several concepts of crossing. 
Two subsets of $\preal$ are \emph{rotationally congruent} if there is a rotation that takes one of them to the other. By a \emph{quasiorder} (also known as \emph{preorder}) we mean a reflexive transitive relation. \emph{Partial orders} are  antisymmetric quasiorders and a \emph{partially ordered set} (also known as a \emph{poset}) is a pair $\tuple{A;\leq}$ such that $A$ is a nonempty set and   $\leq$ is a partial order on $A$.

\begin{definition}
Let $\CrCo$ denote the set of the four concepts of crossing for planar compact convex sets investigated in this paper; the acronym comes from ``Hierarchy of Crossing Concepts''. For $u,v\in \CrCo$, let $u\leq v$ mean that $u$ implies $v$. That is, $u\leq v$ iff for any compact convex subsets $D$ and $L$ of $\preal$, if $D$ crosses $L$ in the sense of $u$, then $D$ crosses $L$ in the sense of $v$. Also, let $u\rotleq v$ mean that for any compact convex subsets $D$ and $L$ of $\preal$, if $L$ is obtained from $D$ by a rotation and   $D$ crosses $L$ in the sense of $u$, then $D$ crosses $L$ in the sense of $v$. Clearly, both $\leq$ and $\rotleq$ are quasiorders on $\CrCo$.  Note that both $u \leq v$ and $u\rotleq v$ mean that $u$, as a set of pairs of compact convex subsets of $\preal$, is a subset of $v$.
\end{definition}

\begin{figure}[ht] 
\centerline
{\includegraphics[scale=1.0]{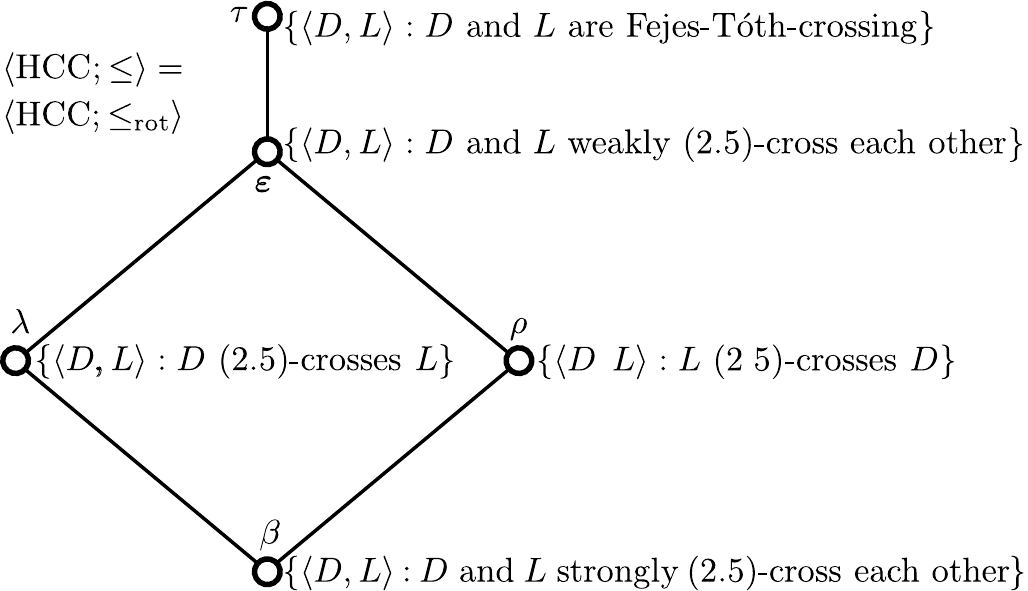}}
\caption{The hierarchy of crossing concepts for compact convex subsets of $\preal$;  $D$ and $L$ stand for compact convex subsets of $\preal$
\label{fignine}}
\end{figure}%

In view of Theorems~\ref{ftthm} and \ref{thmchgcharAB}, the following observation might look a little bit surprising at first sight.

\begin{theorem}[Main Theorem]\label{thmthemain} 
Both $\tuple{\CrCo; \leq}$ and  $\tuple{\CrCo; \rotleq}$ are partially ordered sets, they are the same partially ordered sets, and their common Hasse diagram is the one given in Figure~\textup{\ref{fignine}}. 
\end{theorem}

\section{Lemmas and proofs}
\label{sectCompMn}
We begin this section with an easy lemma.

\begin{lemma}\label{lemmaDiStincT} There exist \emph{rotationally congruent} compact convex subsets $X$ and $Y$ of $\preal$ such that $X$ and $Y$ are \ocrossing{} but they \notwncross.
\end{lemma}

\begin{figure}[ht] 
\centerline
{\includegraphics[scale=1.0]{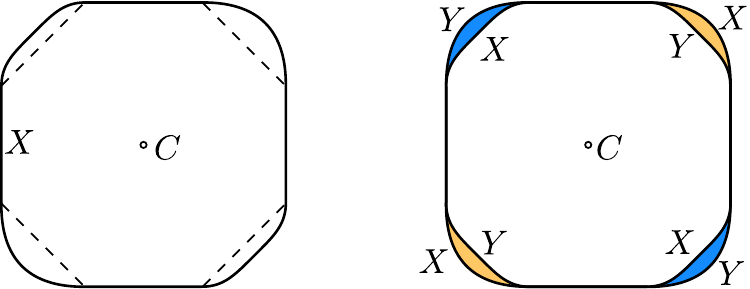}}
\caption{Our construction proving Lemma~\ref{lemmaDiStincT}
\label{figone}}
\end{figure}%

\begin{proof}[Proof of Lemma~\ref{lemmaDiStincT}] Let $X$ be the convex hull of the solid curve given on the left of Figure~\ref{figone}. In order to define it more precisely, consider the graphs $F:=\set{\pair x{f(x)}:-1\leq x\leq 1}$ and $G:=\set{\pair x{g(x)}:-1\leq x\leq 1}$
of the concave real functions $f\colon[-1,1]\to \mathbb R$, $x\mapsto (1-x^2)/2$ and $g\colon[-1,1]\to \mathbb R$, $x\mapsto (1-x^4)/4$, respectively; see Figure~\ref{figtwo}. Both of them are tangent to (the graph of) the absolute value function at $x=-1$ and $x=1$. Next, take a regular octagon. On the left of Figure~\ref{figone}, every second edge of this octagon is given by a dashed line. Replace two opposite dashed edges of the octagon by  congruent copies of $F$, and replace the rest of dashed edges by congruent copies of $G$. So the boundary $\bnd X$ of $X$ consists of four straight line segments, two arcs congruent to $F$, and two arcs congruent to $G$.
(Note at this point that Figure~\ref{figtwo} is scaled differently from Figure~\ref{figone}.)  Next, let rotate $X$  by 90 degrees counterclockwise around the center $C$ of symmetry of the original octagon, and let $Y$ be the compact convex set we obtain in this way. On the right of Figure~\ref{figtwo}, $X\setminus Y$ and $Y\setminus X$ are denoted by light-grey (or yellow) and by dark-grey (or blue) respectively. Clearly, $X$ and $Y$ are \ocrossing. Since $g(x)<f(x)$ for every $x$ from the open interval $(-1,1)$, it follows from our construction that $X$ and $Y$ have exactly four common supporting lines and each of these lines contains one of the non-dashed edges of the initial octagon as an interval. Hence, for every common supporting line $t$, we have that $(X\cup Y)\cap t \subseteq X\cap Y$. Hence, the first point of ($X\cup Y)\cap t$ is in $X\cap Y$ but outside $X\setminus Y$. Therefore, $X$ \notncross{} $Y$. Since the role of $X$ and $Y$ is rotationally symmetric, $Y$ \notncross{} $X$. That is, $X$ and $Y$ \notwncross.
\end{proof}

\begin{figure}[ht] 
\centerline
{\includegraphics[scale=1.0]{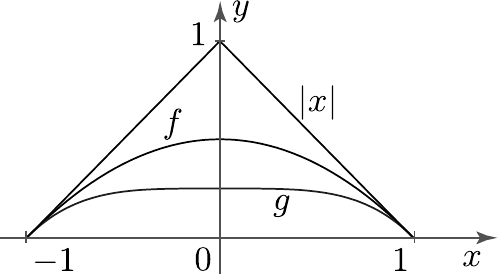}}
\caption{Auxiliary functions for the proof of Lemma~\ref{lemmaDiStincT}
\label{figtwo}}
\end{figure}%

\begin{lemma}[Main Lemma]\label{lemmamain}
Let $D$ and $L$ be nonempty compact convex subsets of the plane $\preal$, then the following two implications hold.
\begin{enumeratei}
\item\label{thmthisa}  If $D$ \ncross{} $L$, then $D$ and $L$ are \ocrossing.
\item\label{thmthisb}  If $D$ and $L$ \wncross{}, then $D$ and $L$ are \ocrossing.
\end{enumeratei}
\end{lemma}

\begin{proof}[Proof of Lemma~\textup{\ref{lemmamain}}] 
Since we are going to rely on continuity,  we recall some terminology and well-known facts; these facts are summarized in  Cz\'edli and Stach\'o~\cite{czgstacho}. It is well known that 
\begin{equation}
\parbox{10.8cm}{if the interior of a compact convex set $X\subseteq \preal$ is nonempty, then its boundary,  $\bnd D$, is a rectifiable Jordan curve of positive finite length.}
\label{eqpbxRcTfJcRvs}
\end{equation} 
A \emph{pointed supporting line} of a compact convex set $H\subseteq \preal$ is a pair $\pair P\lne$ such that $P\in\bnd H$ and $\lne$ is a supporting line of $H$ through $P$;  it is uniquely determined by $\pair{P}{\dir\lne}$, which belongs to the \emph{cylinder} $\cylin := \preal\times\ucirc$.  
We have proved in \cite{czgstacho} that for every compact convex set $H\subseteq \preal$,
\begin{equation}
\parbox{9.6cm}{$\slcr H:=\{\pair P{\dir\lne} : \pair P\lne$ is a pointed supporting line of $H\}$ is a \emph{rectifiable} simple closed curve.}
\label{eqpbxslhDtGnd}
\end{equation}
In  Cz\'edli and Stach\'o~\cite{czgstacho}, we introduced the term \emph{slide-turning} for pointed supporting lines to express the idea that we are \emph{moving} along $\slcr H$. Unless otherwise stated, we always slide-turn  a pointed supporting line $\pair P\lne$  counterclockwise; this means that both $P$ on $\bnd H$ and $\dir\lne$ on $\ucirc$ go counterclockwise. The same convection applies to points, which always move \emph{counterclockwise} unless otherwise stated. 
The visual meaning of \eqref{eqpbxslhDtGnd} is that we can think of slide-turning as a continuous progression in a finite interval of time; this is why the concept of pointed supporting lines has been introduced. 

After these preliminaries,  we  deal with part \eqref{thmthisa} first. So
assume that $D$ and $L$ are nonempty compact convex subsets of the plane such that $D$ \ncross{} $L$. 

First, for the sake of contradiction, suppose that $D$ or $L$ is a singleton $\set P$. Then slide-turning its supporting lines means that we simply turn a directed line through $P$, and it follows trivially that $D$ and $L$ have at most one common supporting lines. This contradicts our assumption that $D$ \ncross{} $L$. Therefore, we conclude that none of $D$ and $L$ is a singleton.

\begin{figure}[ht] 
\centerline
{\includegraphics[scale=1.0]{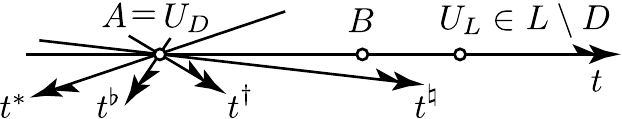}}
\caption{If $D=[A,B]$ is a line segment
\label{figthree}}
\end{figure}%

Second, for the sake of contradiction again, suppose that the interior of  $D$ is empty. Then, since it is not a singleton, $D$ is a line segment with distinct endpoints $A$ and $B$. Suppose that $t$ is a common supporting line of $D$ and $L$ such that $A,B\in t$ and $t$ satisfies the first half of \eqref{eqpbxcspLndf}. 
Choosing the coordinate system appropriately, we can assume that $\dir {t}=0$; see Figure~\ref{figthree}. 
Let  $A<_t B$ (with respect to the orientation of $t$);  otherwise we could change the notation. So we have that $D=[A,B]$. 
Clearly,  $U_D$ from  \eqref{eqpbxcspLndf} is $A$. Using that  $U_L$ from \eqref{eqpbxcspLndf}  is not in $D=[A,B]$,  $D$ is convex, and $U_D<_tU_L$, it follows that $B<U_L$. 
Now, we focus our attention on $t'$ from \eqref{eqpbxcspLndf}. It is distinct from $-t$ since otherwise $U'_L=A=U_D$ would belong to $D$ and this would contradict \eqref{eqpbxcspLndf}. So $t'$ is a supporting line of $D$ with $\dir{t'}\notin\set{0,\pi}$, whereby exactly one of the containments $A\in t'$ and $B\in t'$ holds. 
If $t'$ went through $B$, then $A\in D$ and $U_L\in L$ would be strictly on different sides of $t'$ by $ A<_{t}B<_{t} U_L$,
contradicting \eqref{eqpbxDsPlNdf}. Hence, $t'$ goes through $A$. Since $\dir{t'}\notin\set{0,\pi}$ and since 
$U_L\in L$ is in the left half-plane determined by $t'$, it follows that $\pi<\dir{t'}<2\pi$;  see Figure~\ref{figthree} where $t^\ast$, $t^\flat$, $t^\dag$, and $t^\natural$ indicate some 
possibilities for $t'$. However, then $A=U'_D<_{t'}U'_L$ implies that $U'_L\in L$ is below $t$, that is strictly on the right of $t$, contradicting the fact that $L$ is on the left of $t$. This contradiction shows that no common supporting line of $D$ and $L$ can contain $A$ and $B$, that is, 
\begin{equation}
\parbox{9cm}{$D$ cannot be a subset of $t$ if $t$ satisfies the first half of \eqref{eqpbxcspLndf}.}
\label{eqtxtsdPzbnGjGxB}
\end{equation}
Therefore, still for the case $D=[A,B]$,  it follows that  
\begin{equation}
\parbox{8cm}{every common supporting line satisfying the first half of \eqref{eqpbxcspLndf} contains exactly one of $A$ and $B$.}
\label{eqpbsczhBsmW}
\end{equation}
At present, the role of $A$ and $B$ in \eqref{eqpbsczhBsmW} and that of $t$ and $t'$ is \eqref{eqpbxcspLndf} are symmetric. So \eqref{eqpbsczhBsmW} allows us to assume that  there is a common supporting line $t$ of $D$ and $L$ such that $A\in t$ and $t$ satisfies the first half of \eqref{eqpbxcspLndf}. Then $A=U_D <_t U_L\in L$. We can assume that $A$ and $B$ are on the $x$-axis such that $A <_x B$; see Figure~\ref{figsix}. Then $\pi<\dir t<2\pi$ since $B$ is to the left of $t$. 
For the sake of contradiction, suppose that we can rotate $t$ around $A$ counterclockwise to obtain a common supporting line $t'$ satisfying the second half of  \eqref{eqpbxcspLndf}. By the \emph{positive $A$-ray}  of $t$ we mean the ray $\set{X\in t:  A<_t X}$. 
Similarly, the \emph{negative $A$-ray}  of $t$ is $\set{X\in t:  X <_t A}$; it is often denoted by a dotted ray; see Figure~\ref{figsix}. It is clear by \eqref{eqpbxcspLndf} that $U_L$ is on the positive $A$-ray of $t$. When we rotate $t$ counterclockwise by an angle $\alpha\in (0,2\pi)$, then its positive $A$-ray is also rotated. The point $U'_L$ belongs to the positive $A$-ray of $t'$. The lines $t^\ast$ and $t^\flat$
in Figure~\ref{figsix} and \eqref{eqtxtsdPzbnGjGxB}  indicate that  $\alpha>\pi$ has to hold to obtain $t'$, since otherwise $U_L\in L$ would not be on the left half-plane of $t'$. Furthermore, $t^\natural$ shows that $\pi<\alpha<2\pi$ is impossible, because otherwise the positive $A$-ray of $t'$ is strictly on the right half-plane of $t$ but contains $U'_L\in L$, contradicting the fact that the whole $L$ is on the left of $t$. So $t'$, which is a line through $A$ but distinct from $t$, cannot be obtained from $t$ by rotating it by an angle $\alpha\in(0,2\pi)$, which is impossible.  Hence, we conclude from this contradiction that at most one of $t$ and $t'$ goes through $A$. The same holds for $B$, because $A$ an $B$ play symmetric roles. So, by \eqref{eqpbsczhBsmW}, we can choose the notation so that 
\begin{equation}
\text{For $t$ and $t'$ satisfying \eqref{eqpbxcspLndf},
$A\in t$, $B\notin t$, $A\notin t'$, and $B\in t'$;}
\label{eqtxtsdzvbhBmTN}
\end{equation} 
see Figure~\ref{figseven}.
\begin{figure}[ht] 
\centerline
{\includegraphics[scale=1.0]{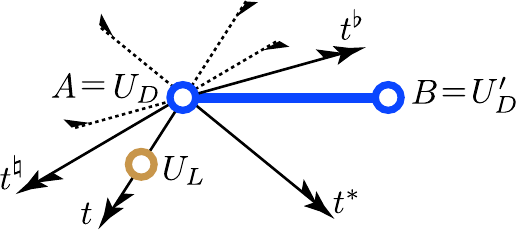}}
\caption{ Illustration for \eqref{eqpbsczhBsmW}
\label{figsix}}
\end{figure}%

We can choose the coordinate system so that $D=[A,B]$ is a horizontal line segment and $A$ and $B$ are on the $x$-axis, $A$ being to the left of $B$; see Figure~\ref{figseven}. Since $B$ is on the left of $t$ and \eqref{eqtxtsdPzbnGjGxB} excludes that $t$ is horizontal, we have that 
$\pi < \dir t <2\pi$, that is, the positive $A$-ray of $t$ is under the $x$-axis. By a similar reason, the positive $B$-ray of $t'$ is above the $x$-axis. Observe that $U_L$ is not on the negative $B$-ray of $t'$ since the first point of $t'\cap(D\cup L)$ with respect to $<_{t'}$ is $B=U'_D$.  
Hence the line segment $[U_L, U'_L]$ intersects $D=[A,B]$ at an inner point $V$. By the convexity of $L$, we have that $V\in L$. But none of $A=U_D$ and $B=U'_D$ is in $L$, whereby it is clear that $D\setminus L$ is not connected. The intersection of the left half-plane of $t$ and that of $t'$ is indicated by (very light) grey in Figure~\ref{figseven}; it can be of a different shape but this does not make a problem. Since $L$ is a subset of this grey area and since $U_L$ and $U'_L$ witness that $L$ contains points below and above the $x$-axes, it follows that $L\setminus D$ is not connected either. Hence, 
$D$ and $L$ are \ocrossing{} if the interior of $D$ is empty.

\begin{figure}[ht] 
\centerline
{\includegraphics[scale=1.0]{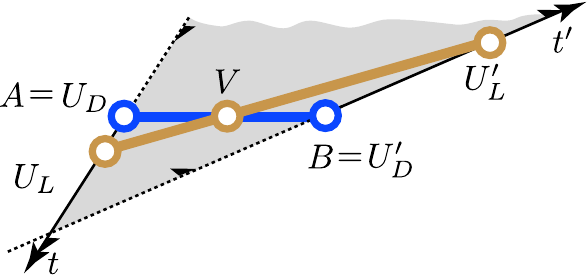}}
\caption{Illustration for \eqref{eqtxtsdzvbhBmTN} \label{figseven}}
\end{figure}%

If the interior of $L$ rather than that of $D$ is empty, then it is easy to modify the argument above to conclude that $D$ and $L$ are \ocrossing{}; the straightforward details are omitted. 

\begin{figure}[ht] 
\centerline
{\includegraphics[scale=1.0]{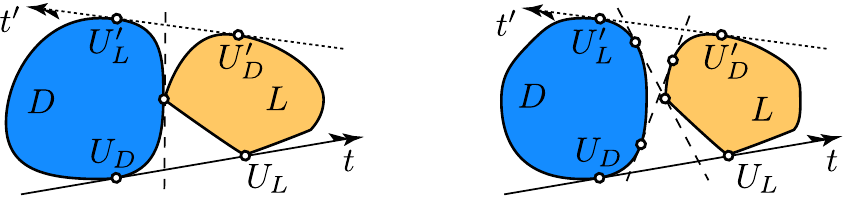}}
\caption{If the interior of $D$ and that of $L$ are disjoint
\label{figfour}}
\end{figure}%

Third, still striving for a contradiction,  for the rest of the proof we suppose that neither the interior of $D$, nor that of $L$ is empty. We claim that 
\begin{equation}
\parbox{7.5cm}{
if the interior of $D$ and that of $L$ are disjoint, then they have only one common supporting line satisfying the first half of condition \eqref{eqpbxcspLndf}.}
\label{eqpbxPTsrenHlsGrn}
\end{equation}
This observation follows from Figure~\ref{figfour}, which carries the generality. By the hyperplane separation theorem, there is at least one non-directed dashed line separating the interior of $D$ and that of $L$; however, such  a line cannot be oriented to satisfy \eqref{eqpbxDsPlNdf}. Furthermore, neither $U'_D\in D\setminus L$, nor $U'_L\in L\setminus D$ holds on the dotted common supporting lines denoted by $t'$, because each of $U'_D$ and  $U'_L$ is in the ``wrong half-plane'' determined by a dashed separating line. Hence,  \eqref{eqpbxPTsrenHlsGrn} follows.

\begin{figure}[ht] 
\centerline
{\includegraphics[scale=1.0]{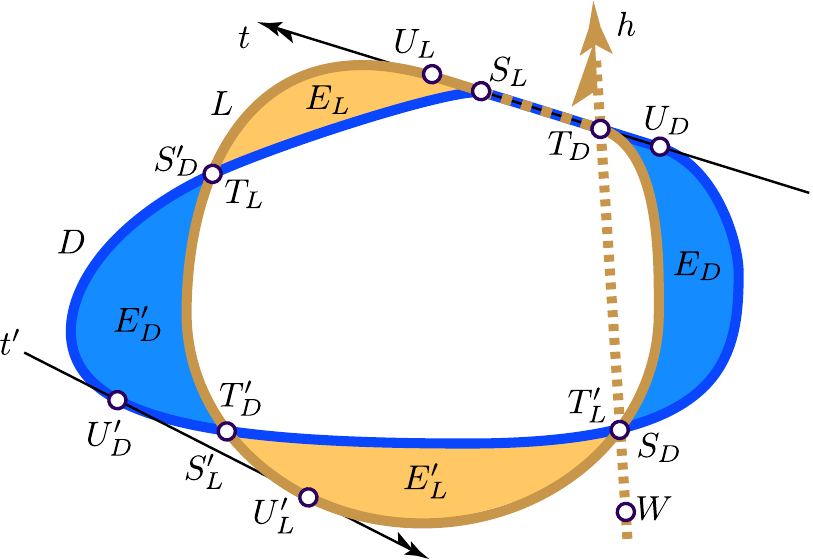}}
\caption{With nonempty interiors,  $D$  \ncross{} $L$ \label{figfive}}
\end{figure}%

Since $D$ \ncross{} $L$, they have at least two common supporting lines   and it follows from  \eqref{eqpbxPTsrenHlsGrn} that the interiors of $D$ and that of $L$ are not disjoint. This implies that the interior of the compact convex set $D\cap L$ is nonempty, whereby \eqref{eqpbxRcTfJcRvs} gives that 
\begin{equation}
\text{$\bnd(D\cap L)$ is a rectifiable Jordan curve of positive finite length.}
\label{eqtxtzhBrtJwfPcG}
\end{equation}

The notation $D$ and $L$ comes from dark-grey and light-grey; which are blue and yellow, respectively, in the colored version of the paper.  Since $D$ \ncross{} $L$, we can pick two common supporting lines $t$ and $t'$ together with the points  occurring in  \eqref{eqpbxcspLndf} such that \eqref{eqpbxcspLndf} holds; see Figure~\ref{figfive}.
Starting from $U_D$ and going on $\bnd D$ clockwise, there is a first point $S_D$ of $\bnd D \cap \bnd L$. Similarly, there is a first point $T_D$ of $\bnd D \cap \bnd L$ if we go counterclockwise. 
We know by \eqref{eqpbxcspLndf} that $U_D\notin\set{S_D,T_D}$.
Analogously, starting from $U_L$ and walking along $\bnd L$ clockwise and counterclockwise, we obtain the first points $S_L$ and $T_L$ of $\bnd L \cap \bnd D$, respectively; see Figure~\ref{figfive}.
It is clear again that $U_L\notin\set{S_L,T_L}$ and $U_D\notin\set{S_D,T_D}$. The points $S_D$, $T_D$ and the arcs between them on the two boundaries define an ``ear'' $E_D$; it is the rightmost  dark region in Figure~\ref{figfive}. This ear is understood so that it does not include the ``light'' arc of $\bnd L$ from $S_D$ and $T_D$ (going counterclockwise). However $E_D$ includes the dark-grey arc from $S_D$ to $T_D$ on $\bnd D$ except for its endpoints $S_D$ and $T_D$. 
 So, $E_D\subseteq D\setminus L$. The points $S_D$ and $T_D$ will be called the \emph{starting point} and the \emph{terminating point} of $E_D$; this explains  $S$ and $T$ in the notation. Similarly, the arcs on the boundaries $\bnd D$ and $\bnd L$ from $S_L$ to $T_L$  form an ear $E_L$; it is the upper light-grey region in the figure, it does not include the ``dark arc'' from $S_L$ to $T_L$ on $\bnd D$ and it is a subset of $L\setminus D$.  The other common supporting line, $t'$, determines the ears $E'_D$ and $E'_L$ and their starting and terminating points $S'_D$, $T'_D$, $S'_L$, $T'_L$ analogously; see Figure~\ref{figfive}.

Some comments on  Figure~\ref{figfive} seem appropriate here. Although $S_L$ is distinct from $T_D$,  the equality $T'_D=S'_L$ indicates that this is not always so. 
As Figure~\ref{figeight} shows, none of the equalities $T_L=S'_D$
and  $T'_L= S_D$  is necessary. Also, 
$T_D \neq S_L$  witnesses that $T'_D = S'_L$ in Figure~\ref{figfive} is not necessary either. Note also that the situation can be much more involved than those in Figures~\ref{figfive} and \ref{figeight}. If we start from an $n$-gon for a large natural number $n$ rather than from an hexagon, then we can easily construct $D$ and $L$  having more than two common supporting lines and more than two ears. Combining this idea with the construction of  Cantor Set, it is not hard to construct compact convex sets $D$ and $L$ that have $\aleph_0$ many ears such that none of these ears has a neighboring ear. The present paper neither needs, nor details this peculiar case, which explains why we  
do not claim that, say, $E'_D$ is next ear after $E_L$ if we go counterclockwise. We claim only the following. 
\begin{equation}
\parbox{7.0cm}{Each of $t$,  $E_D$, and $E_L$ determines the other two. The same holds for $t'$, $E'_D$, and $E'_L$.}
\label{eqpbxRdTrmSpprls}
\end{equation}
By symmetry, it suffices to deal with the first half of \eqref{eqpbxRdTrmSpprls}. Clearly, $t$ determines the ears $E_D$ and $E_L$ by their definitions. Consider the (directed) secant $h$ of $L$ from $S_D$ to $T_D$, it is given by a thick dotted light-grey line in Figure~\ref{figfive}. Let $L^\ast$ be the intersection of $L$ and the closed left half-plane determined by $h$. Since the ear $E_D$ is in the closed right half-plane determined by this secant, so is  $\phullo(E_D)$. 
By the definition of $S_D$ and $T_D$, none of the internal points of the arc of $\bnd L$ between $S_D$ and $T_D$ belongs to $\bnd D$. Hence, going from $T_D$ along $\bnd L$ counterclockwise, $T_L$ is not later then $S_D$, and we conclude that the ear $E_L$ is in the closed left half-plane determined by the secant. In particular, $U_L\in L^\ast$.   The interiors of $\phullo(E_D)$ and $L^\ast$ are disjoint, because they are in opposite half-planes of $h$. 
Thus, applying \eqref{eqpbxPTsrenHlsGrn} to  $L^\ast$  and the convex hull of $E_D$,  we obtain that $E_D$ and $L^\ast$ together determine $t$. But  $E_D$ determines $S_D$, $T_D$ and so $L^\ast$, whereby we conclude that $E_D$ alone determines $t$.
So does $E_L$ by a similar reasoning, or because of (left, counterclockwise)--(right,clockwise) duality.

Next, starting from $S_D$, walk around the rectifiable Jordan curve $\bnd{(D\cap L)}$ until we arrive at $S_D$ again; see \eqref{eqtxtzhBrtJwfPcG}. In other words, we walk \emph{fully around}  $\bnd{(D\cap L)}$. While walking, $E_L$ comes immediately after $E_D$ among the ears. That is, first we walk in the interior of $L$, and the next interior in which we walk is the interior of $D$;
either because $T_D=S_L$, or because the line segment $[T_D, S_L]$ is a subset of $\bnd(D)\cap\bnd (L)$.    Since $t'\neq t$ and \eqref{eqpbxRdTrmSpprls} yield that $E'_D\neq E_D$, it follows that $E'_D$ comes before we reach $S_D$ again. 
So does $E'_L$, since it comes right after $E'_D$ and this part of our walk along $E'_L$ goes in the interior of $L$ while the walk along $E_D$ goes on the boundary of $L$. 
The ears $E_D$ and $E'_D$ will be called the two \emph{$D$-ears} while $E_L$ and $E'_L$ are the two \emph{$L$-ears}. (There can be other ears but we disregard them.)  The argument of this paragraph shows that no matter if we go clockwise or counterclockwise,
\begin{equation}
\parbox{8.5cm}{if we depart from $S_D$ and walk fully around  $\bnd{(D\cap L)}$, then the two $D$-ears alternate with the two $L$-ears.}
\end{equation}
Next, we claim that 
\begin{equation}
\parbox{8.8cm}{The two $D$-ears are connected components of $D\setminus L$ while the two  $L$-ears are connected components of $L\setminus D$.}
\label{eqpbxdhGmsBSv}
\end{equation}
It suffices to deal with $E_D$ since the rest of ears can be handled similarly. Assume that $X\in D\setminus L$ is a point such that there is a continuous curve $g$ within $D\setminus L$ connecting $X$ and a point $Y\in E_D$. 
For the sake of contradiction, suppose that $X\notin E_D$. Then $X$ is in the left half-plane determined by the secant line $h$ while $Y$ is in the right half-plane. By continuity, there is a point $W\in h\cap g$. If $W<_h S_D$,  then $S_D$ is in the interior of the non-degenerate quadrangle formed by four points, $W$, $U_D$, $T_D$, and $S'_D$ of $D$; see Figure~\ref{figfive}. So, since $D$ is convex, $S_D$ has a (small) neighborhood that is a subset of $D$, and this contradicts $S_D\in\bnd D$. 
Replacing $T_D$ by $S_D$, we obtain a similar contradiction if $T_D <_h W$. 
Hence, $S_D\leq_h W\leq_h T_D$, that is, $W\in [S_D,T_D]\subseteq L$, which contradicts $W\in g\subseteq D\setminus L$. This proves 
\eqref{eqpbxdhGmsBSv}. Finally, \eqref{eqpbxdhGmsBSv} yields that $D$ and $L$ are \ocrossing. Thus, part  \eqref{thmthisa} of Lemma~\ref{lemmamain} hold.

Finally, part  \eqref{thmthisb} of Lemma~\ref{lemmamain} follows from part \eqref{thmthisa} and from the fact that \ocrossing{} is a symmetric relation. 
\end{proof}

\begin{figure}[ht] 
\centerline
{\includegraphics[scale=1.0]{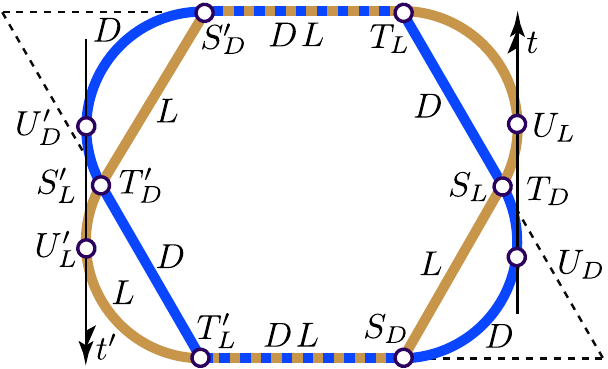}}
\caption{A slightly different arrangement of ears
\label{figeight}}
\end{figure}%

\deletethis{
\begin{align*}
& \tuple{\CrCo; \leq}=\tuple{\CrCo; \rotleq}\cr
&\set{\pair D L: \text{$D$ and $L$ are \ocrossing}} \cr
&\set{\pair D L: \text{$D$ and $L$ \wncross}} \cr
\set{\pair D L: \text{$D$ \ncross{} $L$}} \quad&\quad
\set{\pair D L: \text{$L$ \ncross{} $D$}} \cr
&\set{\pair D L: \text{$D$ and $L$  \sncross}}
\end{align*}
}

\begin{proof}[Proof of Theorem~\ref{thmthemain}]
Let $D$ be the compact convex set we obtain from a regular hexagon with vertices $S_D$, $T_D=S_L$, $T_L$, $S'_D$, $T'_D=S'_L$, and $T_L$ as Figure~\ref{figeight} shows; the notation is borrowed from Figure~\ref{figfive}.
Namely, two opposite edges of the hexagon are replaced by congruent circular arcs that are tangent to the undirected thin dashed lines determined by the neighboring edges. The boundary of $D$ is drawn in dark grey. We obtain $L$ from $D$ by rotating it around the center of the hexagon by $\pi/3$ counterclockwise; $\bnd(L)$ is drawn in light-grey while dark-grey and light-grey alternate on  $\bnd(D)\cap \bnd(L)$. 
 The common supporting lines $t$ and $t'$ witness that 
\begin{equation}
\text{$D$ \ncross{} $L$.}
\label{eqtxtshzTmBw}
\end{equation}
 There are exactly two more common supporting lines $h$ and $h'$; they are horizontal with $\dir h=0$ and $\dir{h'}=\pi$; $h$ and $h'$ are not indicated in the figure.
For each of the four common supporting lines, the first point in the intersection of this supporting line with $L\cup D$ belongs to $D$. Therefore, 
\begin{equation}
\text{$L$ \notncross{} $D$ but they are clearly \ocrossing.}
\label{eqtxtzBmPwRr}
\end{equation}
Denote the elements of $\CrCo$ by 
$\beta$ (bottom), $\lambda$ (left), $\rho$ (right), $\tau$ (top), and $\epsilon$ (else), as shown in Figure~\ref{fignine}. Using \eqref{eqtxtshzTmBw}, \eqref{eqtxtzBmPwRr}, and the fact that we could rename $\pair D L$ to $\pair L D$, we conclude  that $\lambda$ and $\rho$  are  incomparable with respect to $\rotleq$, whence they are also incomparable with respect to $\leq$. For the rest of the proof, note that we need to prove the incomparabilities and the comparabilities only for  $\rotleq$ and only for $\leq$, respectively; we will rely on this remark implicitly.

It is trivial that $\beta\leq \lambda \leq \epsilon$ and $\beta\leq \rho \leq \epsilon$. By Lemma~\ref{lemmamain}, $\epsilon\leq \tau$.
Lemma~\ref{lemmaDiStincT} yields that $\tau\rotnleq \epsilon$.
The pair $\pair D L$ from \eqref{eqtxtshzTmBw} and \eqref{eqtxtzBmPwRr} gives that $\epsilon\rotnleq \rho$ and 
$\lambda\rotnleq \beta$.  
. Hence, after renaming the pair $\pair D L$ to $\pair {L'}{D'}$, we also obtain that $\epsilon\rotnleq \lambda$ and 
$\rho\rotnleq \beta$.  By transitivity, the comparabilities and incomparabilities we have shown above  imply Theorem~\ref{thmthemain}.
\end{proof}

Although ``triviality''  is not a rigorous mathematical concept,
we conclude this section with the following observation.

\begin{observation}\label{observatimpLc}
Theorem~\textup{\ref{ftthm}}, which we cited from Fejes-T\'oth~\cite{fejestoth}, follows trivially from Theorem~\textup{\ref{thmchgcharAB}}, taken from Cz\'edli~\cite{czgcharcircab}, and Theorem~\textup{\ref{thmthemain}}. 
\end{observation}

Although a  true statement is implied by any other statement in principle,  neither Theorem~\textup{\ref{ftthm}}, nor 
Theorem~\textup{\ref{thmthemain}} seems to be useful in the proof of 
Theorem~\textup{\ref{thmchgcharAB}}.

\section{From congruence lattices to the present paper}\label{sectionoutlook}
The purpose of this section is to point out how distant fields of mathematics influenced each other in the progress leading to the present paper. For non-specialists, we mention only that  combinatorics, and geometry, and mainly \emph{lattice theory} occurred among the precursors.
 
The rest of this section is mainly for lattice theorists, 
and even some of them may feel that a part of the concepts below would have deserved definitions. The excuse is that our only purpose is to give a short historical survey to exemplify how certain entirely lattice theoretical problems led to this paper  belonging to geometry; a detailed survey with definitions and theorems would be much longer.

By old results of Funayama and Nakayama~\cite{funanaka}, R.\,P.\ Dilworth (see MathSciNet MR0139551), and Gr\"atzer and Schmidt~\cite{grSch1962}, finite distributive lattices $D$ are, up to isomorphism, exactly the congruence lattices $\Con(L)$ of finite lattices $L$. There are many results stating that $D\cong \Con(L)$ can be achieved by a finite lattice $L$  having ``nice'' properties; see the monograph Gr\"atzer~\cite{gratzerBypicturebook} for a survey.  One of these nice properties is that $L$ is a planar semimodular lattice; this concept was investigated intensively in Gr\"atzer and Knapp~\cite{gratzerknapp1} and \cite{gratzerknapp3},  devoted mostly to the $D\cong \Con(L)$ representation problem. It appeared already in  Gr\"atzer and Knapp~\cite{gratzerknapp1} that the structure of a planar semimodular lattice is well captured by an even more particular lattice, which they called a  \emph{slim planar semimodular lattice}.  (Note that  ``planar'' is automatically understood and so dropped  in some papers.) 

Soon after that  Gr\"atzer and Knapp~\cite{gratzerknapp1} and \cite{gratzerknapp3} made slim semimodular lattices popular, many additional papers started to investigate them; here we mention only Cz\'edli~\cite{czg-mtx},
\cite{czgtrajcolor}, and \cite{czgquasiplanar},
Cz\'edli and Gr\"atzer \cite{czgggresect} and \cite{czgggltsta}, 
Cz\'edli, Ozsv\'art, and Udvari~\cite{czgozsudv}, Cz\'edli and Schmidt~\cite{czgschtJH}, \cite{czgschvisual}, \cite{czgschslim2},  \cite{czgschperm}, and Gr\"atzer~\cite{gratzerBypicturebook}; see also the bibliographic sections of
these papers. In particular, \cite{czgozsudv} deals mainly with slim planar semimodular lattices but has links to group theory and combinatorics. An anonymous referee of  \cite{czgozsudv} pointed out that the lattices from  \cite{czgozsudv} are in close connection with finite convex geometries, which are combinatorial structures.  These structures and equivalent structures  had frequently been discovered by 1985; see Monjardet~\cite{monjardet}. Note that a concept equivalent to that of  finite convex geometries was first discovered within lattice theory; see Dilworth~\cite{dilworthconvgeo} and  Monjardet~\cite{monjardet}. 

Recently, various representation theorems are available for  convex geometries and  for the corresponding lattices; we mention only Adaricheva~\cite{adaricarousel},  Adaricheva and  Cz\'edli \cite{kaczg}, Adaricheva, Gorbunov and Tumanov~\cite{r:adarichevaetal}, Adaricheva and Nation~\cite{kirajbbooksection} and \cite{kajbn},
Cz\'edli~\cite{czgcoord}, Cz\'edli and Kincses~\cite{czgkj},   Kashiwabara,  Nakamura, and Okamoto~\cite{kashiwabaraatalshelling}, and Richter and Rogers~\cite{richterrogers}.  Cz\'edli~\cite{czgcircles} gave a lattice theoretical approach to a new sort of representation, in which some convex geometries were
represented by circles. This paper raised the question which finite convex geometries can be represented. Soon afterwards, Adaricheva and Bolat~\cite{kabolat} proved that not all finite convex geometries; see also Cz\'edli \cite{czgabp} for an alternative proof. The reason of this result is the \emph{Adaricheva-Bolat property}, which is a convex combinatorial property that circles have but most convex geometries do not have. 
Finally, Cz\'edli~\cite{czgcharcircab} proved that the Adaricheva-Bolat property characterizes circles, and \cite{czgcharcircab} is the immediate precursor of the present paper. The question whether ellipses rather than circles are appropriate to represent all finite convex geometries was raised in Cz\'edli~\cite{czgcircles}. This question has recently been answered in negative by Kincses~\cite{kincses}, who presented an Erd\H os-Szekeres type obstruction to such a representation.

\end{document}